\documentclass[leqno,12pt]{article} %leqno is the option to put formula numbers on the left side
\usepackage{amsmath, amssymb}
\usepackage{amsthm} %theorem environment option
\theoremstyle{plain} %text of this environment is typesetted in italics
\newtheorem{theorem}{\indent\sc Theorem}[section]
\newtheorem{lemma}[theorem]{\indent\sc Lemma}
\newtheorem{corollary}[theorem]{\indent\sc Corollary}
\newtheorem{proposition}[theorem]{\indent\sc Proposition}

\theoremstyle{definition} %text of this environment is typesetted in roman letters
\newtheorem{definition}[theorem]{\indent\sc Definition}
\newtheorem{remark}[theorem]{\indent\sc Remark}
\newtheorem{example}[theorem]{\indent\sc Example}

\title{On some special symmetries \\ of a biwarped product-type $3$-manifold}
\author{Adara M. Blaga}

\date{}
\begin{document}

\maketitle

\markboth{{\small\it {\hspace{1cm} On some special symmetries of a biwarped product-type $3$-manifold}}}{\small\it{On some special symmetries of a biwarped product-type $3$-manifold\hspace{1cm}}}

%%%%%%%%%%%%%%% footnote %%%%%%%%%%%%%%%%
\footnote{ %2010 MSC numbers
}
\footnote{ %key words and phrases
}

\begin{abstract}
We investigate special Killing vector fields on $3$-dimensional Riemannian manifolds of biwarped product-type. Starting from a diagonal metric on $\mathbb R^3$ determined by two nontrivial warping functions and a constant scaling factor, we derive the system of equations characterizing Killing fields and provide a description of their structure. Families of solutions are obtained, depending on the expressions and on the relations between the warping functions, including explicit examples of both warped and biwarped product cases. These results continue recent work on symmetries of manifolds with diagonal metrics.
\end{abstract}

\section{Preliminaries}

Killing vector fields play a central role in differential geometry and mathematical physics within the study of Riemannian and pseudo-Riemannian manifolds. By definition, a Killing vector field is a vector field that preserves the metric tensor under the flow it generates. Equivalently, it is a solution of the Killing equation, which expresses the vanishing of the Lie derivative of the metric. This condition formalizes the notion of an infinitesimal isometry: the flow of a Killing vector field moves points on the manifold in such a way that distances and angles (measured by the metric) remain unchanged.

The existence of Killing vector fields is deeply tied to the symmetry structure of the manifold. In Riemannian geometry, they describe continuous groups of isometries, such as rotations and translations in Euclidean space, or the symmetries of a sphere or hyperbolic space. These symmetries allow one to reduce geometric and analytic problems to simpler forms. For instance, the Laplacian or geodesic equations often admit conserved quantities associated with Killing fields, which can facilitate explicit calculations.

In pseudo-Riemannian geometry, especially in Lorentzian manifolds such as those used in general relativity, Killing vector fields acquire a more concrete and profound physical significance. A timelike Killing vector corresponds to time translation symmetry, and leads to the conservation of energy along geodesics. A spacelike Killing vector associated with spatial translations or rotations yields conservation of momentum or angular momentum, respectively. These conservation laws are important in analyzing particle motion, gravitational fields, and spacetime structures. For example, the Schwarzschild spacetime possesses both a timelike and a rotational Killing vector field, reflecting the static and spherically symmetric nature of the black hole solution, and leading directly to conserved energy and angular momentum for test particles.

Moreover, the presence of Killing fields provides insights into the global geometry and topology of a manifold. Their algebra, given by the Lie bracket of vector fields, corresponds to the Lie algebra of the isometry group, which is an invariant of the geometry. In physics, this algebra underpins the classification of spacetimes by symmetry, a key tool in both cosmology and black hole theory. Roughly speaking, Killing vector fields are not only elegant geometric objects but also indispensable tools in applications. They encode symmetry, yield conserved quantities, simplify equations, and reveal deep structural properties of both mathematical spaces and physical models of the universe.

Warped product manifolds occupy a central position in differential geometry and mathematical physics because they provide a method to construct new manifolds with controlled curvature properties. A warped product is built from two Riemannian (or pseudo-Riemannian) manifolds: a base manifold and a fiber manifold, together with a positive smooth function called the warping function. The metric of the product is defined in such a way that the fiber is scaled differently at each point of the base, thereby "warping" the product geometry. This construction generalizes the direct product of manifolds and allows for much richer geometric structures.

One of the primary reasons warped products are important is that they give explicit models of manifolds with desired curvature properties. For example, spheres, hyperbolic spaces, and de Sitter or anti-de Sitter spacetimes can all be realized as warped products. The curvature tensor of a warped product admits an explicit formula in terms of the warping function and the curvatures of the base and fiber, which makes these manifolds highly tractable for both theoretical and applied studies.

Warped products also play a major role in general relativity. Many physically relevant solutions of Einstein's field equations are expressed naturally as warped products. The classical Robertson--Walker spacetimes used in cosmology to model an expanding or contracting universe are warped products, with the scale factor serving as the warping function. Similarly, the Schwarzschild solution, describing the geometry outside a non-rotating spherically symmetric mass, can be interpreted as a warped product between a radial--temporal plane and a $2$-sphere. This perspective allows one to understand the causal and geometric structure of spacetimes in a systematic way.

Beyond relativity, warped product manifolds arise in other areas of mathematics. They appear in the study of submanifold geometry, in the classification of Einstein metrics, and in global analysis, where their special structure allows explicit computations of Laplace and Dirac operators. Furthermore, warped products are useful in geometric flows and comparison geometry, where curvature bounds can be derived or modeled using warped product structures. Basically, warped product manifolds provide a powerful framework for constructing and analyzing spaces with controlled curvature and symmetry. Their applications extend from pure geometry to fundamental models of the physical universe, making them an indispensable tool in both mathematics and physics.

Biwarped product manifolds extend the classical notion of warped products by allowing two distinct warping functions acting on two fiber manifolds over a common base. This construction is a natural generalization of warped products, and it significantly enlarges the class of manifolds that can be studied in both pure and applied geometry.
By introducing two independent warping functions, one gains the ability to model spaces where different geometric or physical components evolve at different rates. This is relevant in the study of curvature properties: explicit formulas for the Riemannian curvature tensor of biwarped products can be obtained in terms of the base, fibers, and warping functions. Such formulas allow us to construct new examples of Einstein manifolds, manifolds with constant scalar curvature, or spaces satisfying other special geometric conditions.

Applications appear mainly in general relativity and cosmology. Biwarped products can be used to describe spacetimes with multiple evolving spatial sections, where the expansion of one sector is governed by one warping function and another sector by a different one. This makes them natural candidates for models of anisotropic cosmological universes, where different spatial dimensions expand at unequal rates. For instance, certain Bianchi-type spacetimes and higher-dimensional cosmological models can be interpreted within the biwarped product framework.

In addition, biwarped product manifolds have applications in theoretical physics beyond relativity, such as string theory and higher-dimensional gravity. In these contexts, extra spatial dimensions often require different scaling behaviors, and biwarped products provide an elegant geometric language to encode such structures.

From a purely mathematical perspective, the study of submanifolds, harmonic maps, and geometric flows on biwarped products is an active area of research. Their structure enables explicit computations that would be intractable in more general settings. Furthermore, biwarped products enrich the classification theory of product-type manifolds and provide new examples in the interplay between geometry and topology. As biwarped product manifolds generalize warped products in a natural way, they offer a new tool for constructing and analyzing spaces with diverse curvature and symmetry properties, and they constitute appropriate models in modern mathematical physics, especially in the geometry of spacetime and higher-dimensional theories.

We briefly recall the definitions of a warped product and biwarped product manifold.

\begin{definition}[Bishop and O'Neill, 1969]
Let $(M_1,g_1)$ and $(M_2,g_2)$ be two pseudo-Riemannian manifolds. The \textit{warped product manifold} $M_1 \!\times_{{f}}\!M_2$ is defined as
$$\left(M_1 \times M_2, \ \pi_{1}^{*}(g_1)+(\pi_1^*(f))^2 \pi_{2}^{*}(g_2)\right),$$
where $\pi_{i}^{*}$ is the pullback map via the canonical projection $\pi_{i}$ from $M_1 \times M_2$ onto $M_i$, for $i\in\{1,2\}$, and $f$ is a smooth positive real function defined on $M_1$ called the \emph{warping function}.
A warped product manifold is said to be \textit{non-trivial} if $f$ is not a constant function. If $f$ is constant, then the manifold is just a direct product manifold (and we call it the trivial case).
\end{definition}

\begin{definition}[N\"{o}lker, 1996]
Let $(M_1,g_1)$, $(M_2,g_2)$, and $(M_3,g_3)$ be three pseudo-Riemannian manifolds. The \textit{biwarped product manifold} $M_1 \!\times_{{f_1}}\!M_2 \!\times_{{f_2}}\!M_3$ is defined as
$$\left(M_1 \times M_2\times M_3, \ \pi_{1}^{*}(g_1)+(\pi_1^*(f_1))^2 \pi_{2}^{*}(g_2)+(\pi_1^*(f_2))^2 \pi_{3}^{*}(g_3)\right), $$
where $\pi_{i}^{*}$ is the pullback map via the canonical projection $\pi_{i}$ from $M_1 \times M_2\times M_3$ onto $M_i$, for $i\in\{1,2,3\}$, and $f_1$ and $f_2$ are two smooth positive real functions defined on $M_1$ called the \emph{warping functions}.
If only one of $f_1$ and $f_2$ is constant, then the manifold is a warped product manifold. Moreover, if both $f_1$ and $f_2$ are constant, then the manifold is a direct product manifold (and we call it the trivial case).
\end{definition}

The aim of the present paper is to determine certain symmetries of $\mathbb R^3$ endowed with a Riemannian metric that slightly extends the biwarped product metric, completing some results from \cite{bl24} and \cite{adara}. Moreover, we provide examples of Killing vector fields on a warped and biwarped product $3$-dimensional manifold.

\section{Killing vector fields}

We consider now a Riemannian metric ${g}$ on $\mathbb R^3$ given by
\begin{equation*}
{g}=\frac{1}{f_1^2}dx^1\otimes dx^1+\frac{1}{f_2^2}dx^2\otimes dx^2+\frac{1}{k_3^2}dx^3\otimes dx^3,
\end{equation*}
where $x^1,x^2,x^3$ stand for the standard coordinates in $\mathbb R^3$, $f_1$ and $f_2$ are smooth functions nowhere zero on $\mathbb R^3$ depending only on $x^3$, and $k_3\in \mathbb R\setminus \{0\}$. Let
$$\Big\{E_1:=f_1\frac{\partial}{\partial x^1}, \ \ E_2:=f_2\frac{\partial}{\partial x^2}, \ \ E_3:=k_3\frac{\partial}{\partial x^3}\Big\}$$
be a local orthonormal frame.
Then, the Levi-Civita connection $\nabla$ of $g$ is given by (see \cite{balr}):
$$\nabla_{E_1}E_1=k_3\frac{f_1'}{f_1}E_3, \ \ \nabla_{E_2}E_2=k_3\frac{f_2'}{f_2}E_3, \ \ \nabla_{E_3}E_3=0,$$
$$\nabla_{E_1}E_2=0, \ \ \nabla_{E_2}E_3=-k_3\frac{f_2'}{f_2}E_2, \ \ \nabla_{E_3}E_1=0,$$
$$\nabla_{E_1}E_3=-k_3\frac{f_1'}{f_1}E_1, \ \ \nabla_{E_3}E_2=0, \ \ \nabla_{E_2}E_1=0.$$

We recall that a vector field $V$ tangent to $\mathbb R^3$ is called a \textit{Killing vector field} \cite{killing} if the Lie derivative $\pounds$ of the metric $g$ in the direction of $V$ vanishes, i.e.,
$$(\pounds_Vg)(X,Y):=V(g(X,Y))-g([V,X],Y)-g(X,[V,Y])=0$$
for any tangent vector fields $X,Y$ to $\mathbb R^3$.

Let $V=\sum_{k=1}^3V^kE_k$, with $V^k$, $k\in \{1,2,3\}$, smooth functions on $\mathbb R^3$. Then,
\begin{align*}
(\pounds_Vg)(E_i,E_j)&=E_i(V^j)+E_j(V^i)+\sum_{k=1}^3V^k\{g(\nabla_{E_i}E_k,E_j)+g(E_i,\nabla_{E_j}E_k)\}
\end{align*}
for any $i,j\in \{1,2,3\}$, which is equivalent to the following system
\begin{equation}\label{s0}
\left\{
    \begin{aligned}
&(\pounds_Vg)(E_1,E_1)=2\left\{E_1(V^1)-k_3\frac{f_1'}{f_1}V^3\right\}\\
&(\pounds_Vg)(E_2,E_2)=2\left\{E_2(V^2)-k_3\frac{f_2'}{f_2}V^3\right\}\\
&(\pounds_Vg)(E_3,E_3)=2E_3(V^3)\\
&(\pounds_Vg)(E_1,E_2)=E_1(V^2)+E_2(V^1)\\
&(\pounds_Vg)(E_2,E_3)=E_2(V^3)+E_3(V^2)+k_3\frac{f_2'}{f_2}V^2\\
&(\pounds_Vg)(E_3,E_1)=E_3(V^1)+E_1(V^3)+k_3\frac{f_1'}{f_1}V^1
    \end{aligned}
  \right. \ ,
\end{equation}
and we have

\begin{lemma}
If $f_1=f_1(x^3)$, $f_2=f_2(x^3)$, $f_3=k_3\in \mathbb R\setminus\{0\}$, then the vector field $V=\sum_{k=1}^3V^kE_k$ is a Killing vector field on $(\mathbb R^3,g)$ if and only if
\begin{equation}\label{sms}
\left\{
    \begin{aligned}
&f_1\frac{\displaystyle \partial V^1}{\displaystyle \partial x^1}=k_3\frac{\displaystyle f_1'}{\displaystyle f_1}V^3\\
&f_2\frac{\displaystyle \partial V^2}{\displaystyle \partial x^2}=k_3\frac{\displaystyle f_2'}{\displaystyle f_2}V^3\\
&\frac{\displaystyle \partial V^3}{\displaystyle \partial x^3}=0\\
&f_1\frac{\displaystyle \partial V^2}{\displaystyle \partial x^1}=-f_2\frac{\displaystyle \partial V^1}{\displaystyle \partial x^2}\\
&f_2\frac{\displaystyle \partial V^3}{\displaystyle \partial x^2}+k_3\frac{\displaystyle \partial V^2}{\displaystyle \partial x^3}+k_3\frac{\displaystyle f_2'}{\displaystyle f_2}V^2=0\\
&k_3\frac{\displaystyle \partial V^1}{\displaystyle \partial x^3}+f_1\frac{\displaystyle \partial V^3}{\displaystyle \partial x^1}+k_3\frac{\displaystyle f_1'}{\displaystyle f_1}V^1=0
    \end{aligned}
  \right. \ .
\end{equation}
\end{lemma}

\begin{proposition}\label{pj1}
Let $f_1=f_1(x^3)$, $f_2=f_2(x^3)$, $f_3=k_3\in \mathbb R\setminus\{0\}$. Then, a vector field $V=\sum_{k=1}^3V^kE_k$ is a Killing vector field on $(\mathbb R^3,g)$ if and only if
one of the following four assertions holds:

(i) \begin{equation*}
V^3=c\in \mathbb R;
\end{equation*}

(ii) \begin{equation*}
\left\{
    \begin{aligned}
&V^3=V^3(x^1)\\
&(V^3)'\neq 0\\
&\frac{\displaystyle f_1f_2'}{\displaystyle f_2^2}=c\in \mathbb R
    \end{aligned}
  \right. \ ;
\end{equation*}

(iii) \begin{equation*}
\left\{
    \begin{aligned}
&V^3=V^3(x^2)\\
&(V^3)'\neq 0\\
&\frac{\displaystyle f_1'f_2}{\displaystyle f_1^2}=c\in \mathbb R
    \end{aligned}
  \right. \ ;
\end{equation*}

(iv) \begin{equation*}
\left\{
    \begin{aligned}
&V^3=V^3(x^1,x^2)\\
&\frac{\displaystyle \partial V^3}{\displaystyle \partial x^1}\neq 0, \ \frac{\displaystyle \partial V^3}{\displaystyle \partial x^2}\neq 0\\
&\frac{\displaystyle f_1f_2'}{\displaystyle f_2^2}=c_1\in \mathbb R, \ \frac{\displaystyle f_1'f_2}{\displaystyle f_1^2}=c_2\in \mathbb R
    \end{aligned}
  \right. \ .
\end{equation*}
\end{proposition}
\begin{proof}
\eqref{sms} follows immediately from \eqref{s0}.

The 3rd equation of \eqref{sms} implies that $V^3=V^3(x^1,x^2)$. Expressing now its derivatives from the 6th and the 5th equations of \eqref{sms}, we infer that
\begin{equation}\label{s1}
\left\{
    \begin{aligned}
&\frac{\displaystyle \partial V^3}{\displaystyle \partial x^1}(x^1,x^2)=-k_3\left(\frac{\displaystyle f_1'}{\displaystyle f_1^2}\right)(x^3)V^1(x^1,x^2,x^3)-k_3\frac{\displaystyle 1}{\displaystyle f_1(x^3)}\frac{\displaystyle \partial V^1}{\displaystyle \partial x^3}(x^1,x^2,x^3)\\
&\frac{\displaystyle \partial V^3}{\displaystyle \partial x^2}(x^1,x^2)=-k_3\left(\frac{\displaystyle f_2'}{\displaystyle f_2^2}\right)(x^3)V^2(x^1,x^2,x^3)-k_3\frac{\displaystyle 1}{\displaystyle f_2(x^3)}\frac{\displaystyle \partial V^2}{\displaystyle \partial x^3}(x^1,x^2,x^3)
    \end{aligned}
  \right. \ .
\end{equation}
Now, differentiating the 1st relation from \eqref{s1} with respect to $x^2$, the 2nd one with respect to $x^1$, equalizing them, and using the 4th equation of \eqref{sms}, we get
$$\frac{\displaystyle \partial}{\displaystyle \partial x^3}\left(f_2\frac{\displaystyle \partial V^1}{\displaystyle \partial x^2}(x^1,x^2,\cdot)\right)(x^3)=0,$$
which implies that
\begin{equation}\label{s2}
f_2(x^3)\frac{\displaystyle \partial V^1}{\displaystyle \partial x^2}(x^1,x^2,x^3)=K(x^1,x^2),
\end{equation}
where $K=K(x^1,x^2)$, and further that
\begin{equation}\label{s3}
f_1(x^3)\frac{\displaystyle \partial V^2}{\displaystyle \partial x^1}(x^1,x^2,x^3)=-K(x^1,x^2),
\end{equation}
by means of the 4th equation of \eqref{sms}.
Differentiating \eqref{s2} with respect to $x^1$, \eqref{s3} with respect to $x^2$, and using the 1st and the 2nd equations of \eqref{sms}, we obtain
\begin{equation}\label{s4}
\left\{
    \begin{aligned}
&\frac{\displaystyle \partial K}{\displaystyle \partial x^1}(x^1,x^2)=k_3\left(\frac{\displaystyle f_1'f_2}{\displaystyle f_1^2}\right)(x^3)\frac{\displaystyle \partial V^3}{\displaystyle \partial x^2}(x^1,x^2)\\
&\frac{\displaystyle \partial K}{\displaystyle \partial x^2}(x^1,x^2)=-k_3\left(\frac{\displaystyle f_1f_2'}{\displaystyle f_2^2}\right)(x^3)\frac{\displaystyle \partial V^3}{\displaystyle \partial x^1}(x^1,x^2)
    \end{aligned}
  \right. \ ,
\end{equation}
and we deduce the following four possible cases:
\begin{equation}\label{s5}
\left\{
    \begin{aligned}
&\frac{\displaystyle \partial V^3}{\displaystyle \partial x^1}(x^1,x^2)=0\\
&\frac{\displaystyle \partial V^3}{\displaystyle \partial x^2}(x^1,x^2)=0
    \end{aligned}
  \right. \ ;
\end{equation}
\begin{equation}\label{s6}
\left\{
    \begin{aligned}
&\frac{\displaystyle \partial V^3}{\displaystyle \partial x^1}(x^1,x^2)=0\\
&\frac{\displaystyle \partial V^3}{\displaystyle \partial x^2}(x^1,x^2)\neq 0 \\
&\frac{\displaystyle f_1'f_2}{\displaystyle f_1^2}=c_1\in \mathbb R
    \end{aligned}
  \right. \ ;
\end{equation}
\begin{equation}\label{s7}
\left\{
    \begin{aligned}
&\frac{\displaystyle \partial V^3}{\displaystyle \partial x^1}(x^1,x^2) \neq 0\\
&\frac{\displaystyle \partial V^3}{\displaystyle \partial x^2}(x^1,x^2)= 0 \\
&\frac{\displaystyle f_1f_2'}{\displaystyle f_2^2}=c_2\in \mathbb R
    \end{aligned}
  \right. \ ;
\end{equation}
\begin{equation}\label{s8}
\left\{
    \begin{aligned}
&\frac{\displaystyle \partial V^3}{\displaystyle \partial x^1}(x^1,x^2)\neq 0\\
&\frac{\displaystyle \partial V^3}{\displaystyle \partial x^2}(x^1,x^2)\neq 0 \\
&\frac{\displaystyle f_1'f_2}{\displaystyle f_1^2}=c_1\in \mathbb R\\
&\frac{\displaystyle f_1f_2'}{\displaystyle f_2^2}=c_2\in \mathbb R
    \end{aligned}
  \right. \ .
\end{equation}

By a direct computation, we find that the converse implication also holds true, hence, the proof is complete.
\end{proof}

\begin{remark}
Some examples of non-constant real-valued functions $f_1$ and $f_2$ that satisfy the condition
$$\frac{f_1'f_2}{f_1^2}=c \in \mathbb R\setminus \{0\}$$
are: (i) $f_1(t)=e^t$, $f_2(t)=ce^t$; (ii) $f_1(t)=t$, $f_2(t)=ct^2$ (on an open interval not containing $0$); (iii) $f_1(t)=\sin(t)$, $f_2(t)=c\displaystyle\frac{\sin^2(t)}{\cos(t)}$ (on an open interval not containing $k\pi$ nor $\displaystyle\frac{\pi}{2}+k\pi$ for any integer number $k$).
\end{remark}

\begin{remark}
Under the hypotheses of Proposition \ref{pj1}, we notice that $V^3$ can not depend on $x^3$.
\end{remark}

Now, we will consider the cases when one of the component functions of the vector field is constant, and we prove the following results.

\begin{theorem}\label{tea}
Let $f_1=f_1(x^3)$, $f_2=f_2(x^3)$, $f_3=k_3\in \mathbb R\setminus\{0\}$. Then, a vector field $V=\sum_{k=1}^3V^kE_k$ with $V^1=c_1\in \mathbb R$ is a Killing vector field on $(\mathbb R^3,g)$ if and only if one of the following assertions holds:

(A) \begin{equation*}
\left\{
    \begin{aligned}
&V^1=0\\
&V^2(x^3)=\frac{\displaystyle c}{\displaystyle f_2(x^3)}\\
&V^3=0
\end{aligned}
  \right. \ , \ c\in \mathbb R;
\end{equation*}

(B) \begin{equation*}
\left\{
    \begin{aligned}
&V^1=c_1\\
&V^2(x^3)=\frac{\displaystyle c_2}{\displaystyle f_2(x^3)}\\
&V^3=0
\end{aligned}
  \right. \ , \ c_1\in \mathbb R\setminus\{0\},c_2\in \mathbb R,
\end{equation*}
and $f_1=k_1\in \mathbb R\setminus\{0\}$;

(C) $f_1=k_1\in \mathbb R\setminus\{0\}$, $\frac{\displaystyle f_2''f_2-(f_2')^2}{\displaystyle f_2^4}=k\in \mathbb R$, and,
according to the sign of $k$, we consequently have:

\hspace{0.5cm} (a) $k=0$ and 
\begin{equation*}
\left\{
    \begin{aligned}
&V^1=c_1\\
&V^2(x^2,x^3)=\frac{\displaystyle k_3\bar{c}}{\displaystyle \tilde{c}}(ax^2+b)e^{-\bar{c}x^3}\\
&V^3=a
\end{aligned}
  \right. \ , \ a,b,c_1,\bar{c}\in \mathbb R, \tilde{c}\in \mathbb R\setminus\{0\}, 
\end{equation*}

\hspace{0.5cm} (b) $k<0$ and 
\begin{equation*}
\left\{
    \begin{aligned}
&V^1=c_1\\
&V^2(x^2,x^3)=-\frac{\displaystyle \sqrt{-k}}{\displaystyle k}\left(\frac{\displaystyle f_2'}{\displaystyle f_2^2}\right)(x^3)\left(a_1e^{k_3\sqrt{-k}x^2}-a_2e^{-k_3\sqrt{-k}x^2}\right)\\
&V^3(x^2)=a_1e^{k_3\sqrt{-k}x^2}+a_2e^{-k_3\sqrt{-k}x^2}
\end{aligned}
  \right. \ , \ a_1,a_2,c_1\in \mathbb R,
\end{equation*}

\hspace{0.5cm} (c) $k>0$ and 
\begin{equation*}
\left\{
    \begin{aligned}
&V^1=c_1\\
&V^2(x^2,x^3)=\frac{\displaystyle \sqrt{k}}{\displaystyle k}\left(\frac{\displaystyle f_2'}{\displaystyle f_2^2}\right)(x^3)\left(a_1\sin({k_3\sqrt{k}x^2})-a_2\cos({k_3\sqrt{k}x^2})\right)\\
&V^3(x^2)=a_1\cos({k_3\sqrt{k}x^2})+a_2\sin({k_3\sqrt{k}x^2})
\end{aligned}
  \right. \ , \ a_1,a_2,c_1\in \mathbb R.
\end{equation*}
\end{theorem}
\begin{proof}
In this case, \eqref{sms} becomes
\begin{equation}\label{sms1}
\left\{
    \begin{aligned}
&f_1'V^3=0\\
&\frac{\displaystyle \partial V^2}{\displaystyle \partial x^2}=k_3\frac{\displaystyle f_2'}{\displaystyle f_2^2}V^3\\
&\frac{\displaystyle \partial V^3}{\displaystyle \partial x^3}=0\\
&\frac{\displaystyle \partial V^2}{\displaystyle \partial x^1}=0\\
&f_2\frac{\displaystyle \partial V^3}{\displaystyle \partial x^2}+k_3\frac{\displaystyle \partial V^2}{\displaystyle \partial x^3}+k_3\frac{\displaystyle f_2'}{\displaystyle f_2}V^2=0\\
&\frac{\displaystyle \partial V^3}{\displaystyle \partial x^1}=-k_3c_1\frac{\displaystyle f_1'}{\displaystyle f_1^2}
    \end{aligned}
  \right. \ .
\end{equation}
From the 4th and the 3rd equations of \eqref{sms1}, we deduce that
$$V^2=V^2(x^2,x^3), \ \ V^3=V^3(x^1,x^2),$$
and the 1st equation of the same system implies that either ($i_1$) ($V^3=0$) or ($i_2$) ($f_1=k_1\in \mathbb R\setminus\{0\}$).

($i_1$) If $V^3=0$, from \eqref{sms1} we get
\begin{equation}\label{sms1a}
\left\{
    \begin{aligned}
&V^2=V^2(x^3)\\
&(f_2V^2)'=0\\
&c_1f_1'=0
    \end{aligned}
  \right. \ ;
\end{equation}
therefore, $V^2(x^3)=\frac{\displaystyle c_2}{\displaystyle f_2(x^3)}$, $c_2\in \mathbb R$, and we deduce the following possible cases: (${i_1}_a$) ($c_1=0$) and (${i_1}_b$) ($c_1\neq 0$ and $f_1=k_1\in \mathbb R\setminus\{0\}$).

In the 1st case, (${i_1}_a$), we have $V^1=0$.

In the 2nd case, (${i_1}_b$), we have $V^1=c_1\in \mathbb R\setminus\{0\}$ and $f_1=k_1\in \mathbb R\setminus\{0\}$.

($i_2$) If $f_1=k_1\in \mathbb R\setminus\{0\}$, from \eqref{sms1} we get
\begin{equation}\label{sms1b}
\left\{
    \begin{aligned}
&V^3=V^3(x^2)\\
&\frac{\displaystyle \partial V^2}{\displaystyle \partial x^2}=k_3\frac{\displaystyle f_2'}{\displaystyle f_2^2}V^3\\
&\frac{\displaystyle \partial V^2}{\displaystyle \partial x^3}=-\frac{\displaystyle f_2'}{\displaystyle f_2}V^2-\frac{f_2}{k_3}(V^3)'
    \end{aligned}
  \right. \ .
\end{equation}
By derivating the 2nd and the 3rd equations of \eqref{sms1b} with respect to $x^3$ and $x^2$ respectively, and equalizing them, we get
\begin{equation}\label{cc}
(V^3)''(x^2)=-k_3^2\left(\frac{\displaystyle 1}{\displaystyle f_2}\cdot\left(\frac{\displaystyle f_2'}{\displaystyle f_2^2}\right)'+\left(\frac{\displaystyle f_2'}{\displaystyle f_2^2}\right)^2\right)(x^3)V^3(x^2),
\end{equation}
and we deduce the following possible cases: (${i_2}_a$) ($V^3=0$) and (${i_2}_b$) ($\frac{\displaystyle 1}{\displaystyle f_2}\cdot\left(\frac{\displaystyle f_2'}{\displaystyle f_2^2}\right)'+\left(\frac{\displaystyle f_2'}{\displaystyle f_2^2}\right)^2=k\in \mathbb R$).

In the 1st case, (${i_2}_a$), we get
\begin{equation*}
\left\{
    \begin{aligned}
&V^2=V^2(x^3)\\
&(f_2V^2)'=0
    \end{aligned}
  \right. \ ;
\end{equation*}
therefore, $V^2(x^3)=\frac{\displaystyle c_2}{\displaystyle f_2(x^3)}$, $c_2\in \mathbb R$.

In the 2nd case, (${i_2}_b$), we have
\begin{equation*}
\left\{
    \begin{aligned}
&\frac{\displaystyle 1}{\displaystyle f_2}\cdot\left(\frac{\displaystyle f_2'}{\displaystyle f_2^2}\right)'+\left(\frac{\displaystyle f_2'}{\displaystyle f_2^2}\right)^2=k\\
&(V^3)''=-k_3^2kV^3
    \end{aligned}
  \right. \ .
\end{equation*}
The 1st condition is equivalent to
$$\frac{\displaystyle f_2''f_2-(f_2')^2}{\displaystyle f_2^4}=k.$$
From the 2nd equation, we deduce the following possible cases: 

(${{i_2}_b}_1$) ($k=0$, hence, $V^3(x^2)=a_1x^2+a_2$, $a_1,a_2\in \mathbb R$), 

(${{i_2}_b}_2$) ($k<0$, hence, $V^3(x^2)=a_1e^{k_3\sqrt{-k}x^2}+a_2e^{-k_3\sqrt{-k}x^2}$, $a_1,a_2\in \mathbb R$), 

(${{i_2}_b}_3$) ($k>0$, hence, $V^3(x^2)=a_1\cos(k_3\sqrt{k}x^2)+a_2\sin(k_3\sqrt{k}x^2)$, $a_1,a_2\in \mathbb R$), \\
and $V^2$ satisfies
\begin{equation}\label{s4}
\left\{
    \begin{aligned}
&\frac{\displaystyle \partial V^2}{\displaystyle \partial x^2}=k_3\frac{\displaystyle f_2'}{\displaystyle f_2^2}V^3\\
&\frac{\displaystyle \partial V^2}{\displaystyle \partial x^3}=-\frac{\displaystyle f_2'}{\displaystyle f_2}V^2-\frac{f_2}{k_3}(V^3)'
    \end{aligned}
  \right. \ .
\end{equation}
From the 2nd equation of \eqref{s4} we infer that
$$\frac{\displaystyle \partial }{\displaystyle \partial x^3}\left(f_2V^2(x^2,\cdot)\right)(x^3)=-\frac{\displaystyle f_2^2(x^3)}{\displaystyle k_3}(V^3)'(x^2)
$$
and we obtain
$$V^2(x^2,x^3)=-\frac{\displaystyle F_2(x^3)}{\displaystyle k_3f_2(x^3)}(V^3)'(x^2)+\frac{\displaystyle 1}{\displaystyle f_2(x^3)}M_2(x^2),$$
where $M_2=M_2(x^2)$. Now, using the 1st equation of \eqref{s4}, we find
$$M_2'(x^2)=k_3\left(\frac{\displaystyle f_2'}{\displaystyle f_2}-kF_2\right)(x^3)V^3(x^2),$$
and we deduce the following possible cases: (${{i_2}_b}_a$) ($V^3=0$) and (${{i_2}_b}_b$) ($\frac{\displaystyle f_2'}{\displaystyle f_2}-kF_2=k_0\in \mathbb R$).

In the 1st case, (${{i_2}_b}_a$), we get $M_2=c_2\in \mathbb R$, hence, $V^2(x^3)=\frac{\displaystyle c_2}{\displaystyle f_2(x^3)}$.

In the 2nd case, (${{i_2}_b}_b$), we have $kF_2=\frac{\displaystyle f_2'}{\displaystyle f_2}-k_0$, and
$M_2'(x^2)=k_3k_0V^3(x^2)$. We obtain, by integration, the expression of $M_2$, hence, of $V^2$, according as $k=0$, $k<0$, and $k>0$.

By a direct computation, we find that the converse implication also holds true, hence, the proof is complete.
\end{proof}

\begin{example}
The vector field $V=c\frac{\displaystyle \partial}{\displaystyle \partial x^2}$, $c\in\mathbb R\setminus\{0\}$, 
is a Killing vector field on the biwarped product manifold
$$\left(\mathbb R^3, \ {g}=\frac{1}{f_1^2}dx^1\otimes dx^1+\frac{1}{f_2^2}dx^2\otimes dx^2+dx^3\otimes dx^3\right),$$
where $$f_1(x^3)=e^{a_1x^3}, \ \ f_2(x^3)=e^{a_2x^3}, \ \ a_1,a_2\in \mathbb R\setminus\{0\}, a_1\neq a_2.$$
\end{example}

\begin{remark}
The condition $\frac{\displaystyle f''f-(f')^2}{\displaystyle f^4}=k\in \mathbb R\setminus\{0\}$ from Theorem \ref{tea} is satisfied if and only if the function $f$ is given by
$$f(t)=\tilde{c}e^{\bar{c}t}, \ \ \tilde{c}\in \mathbb R\setminus\{0\}, \bar{c}\in \mathbb R.$$
\end{remark}

\begin{corollary}\label{co2a}
If $f_1=f_2=:f(x^3)$, $f_3=k_3\in \mathbb R\setminus\{0\}$, then $V=\sum_{k=1}^3V^kE_k$ with $V^1=c_1$ is a Killing vector field on $(\mathbb R^3,g)$ if and only if one of the following assertions holds:

(A) \begin{equation*}
\left\{
    \begin{aligned}
&V^1=0\\
&V^2(x^3)=\frac{\displaystyle c}{\displaystyle f(x^3)}\\
&V^3=0
\end{aligned}
  \right. \ , \ c\in \mathbb R;
\end{equation*}

(B) \begin{equation*}
\left\{
    \begin{aligned}
&V^1=c_1\\
&V^2(x^3)=\frac{\displaystyle c_2}{\displaystyle f(x^3)}\\
&V^3=0
\end{aligned}
  \right. \ , \ c_1\in \mathbb R\setminus\{0\},c_2\in \mathbb R,
\end{equation*}
and $f=k_1\in \mathbb R\setminus\{0\}$;

(C) $f=k_1\in \mathbb R\setminus\{0\}$ and
\begin{equation*}
\left\{
    \begin{aligned}
&V^1=c_1\\
&V^2(x^3)=-\frac{\displaystyle k_1a_1}{\displaystyle k_3}x^3+a_2\\&V^3(x^2)=a_1x^2+a_3
\end{aligned}
  \right. \ , \ c_1,a_1,a_2,a_3\in \mathbb R.
\end{equation*}
\end{corollary}

From the symmetry in $V^1$ and $V^2$ of the system \eqref{sms}, we can further conclude.
\begin{theorem}\label{teb}
If $f_1=f_1(x^3)$, $f_2=f_2(x^3)$, $f_3=k_3\in \mathbb R\setminus\{0\}$, then, a vector field $V=\sum_{k=1}^3V^kE_k$ with $V^2=c_2\in \mathbb R$ is a Killing vector field on $(\mathbb R^3,g)$ if and only if one of the following assertions holds:

(A) \begin{equation*}
\left\{
    \begin{aligned}
&V^1(x^3)=\frac{\displaystyle c}{\displaystyle f_1(x^3)}\\
&V^2=0\\
&V^3=0
\end{aligned}
  \right. \ , \ c\in \mathbb R;
\end{equation*}

(B) \begin{equation*}
\left\{
    \begin{aligned}
&V^1(x^3)=\frac{\displaystyle c_1}{\displaystyle f_1(x^3)}\\
&V^2=c_2\\
&V^3=0
\end{aligned}
  \right. \ , \ c_1\in \mathbb R,c_2\in \mathbb R\setminus\{0\},
\end{equation*}
and $f_2=k_2\in \mathbb R\setminus\{0\}$;

(C) $\frac{\displaystyle f_1''f_1-(f_1')^2}{\displaystyle f_1^4}=k\in \mathbb R$, $f_2=k_2\in \mathbb R\setminus\{0\}$, and,
according to the sign of $k$, we consequently have:

\hspace{0.5cm} (a) $k=0$ and 
\begin{equation*}
\left\{
    \begin{aligned}
&V^1(x^1,x^3)=\frac{\displaystyle k_3\bar{c}}{\displaystyle \tilde{c}}(ax^1+b)e^{-\bar{c}x^3}\\
&V^2=c_2\\
&V^3=a
\end{aligned}
  \right. \ , \ a,b,c_2,\bar{c}\in \mathbb R, \tilde{c}\in \mathbb R\setminus\{0\},
\end{equation*}

\hspace{0.5cm} (b) $k<0$ and 
\begin{equation*}
\left\{
    \begin{aligned}
&V^1(x^1,x^3)=-\frac{\displaystyle \sqrt{-k}}{\displaystyle k}\left(\frac{\displaystyle f_1'}{\displaystyle f_1^2}\right)(x^3)\left(a_1e^{k_3\sqrt{-k}x^1}-a_2e^{-k_3\sqrt{-k}x^1}\right)\\
&V^2=c_2\\
&V^3(x^1)=a_1e^{k_3\sqrt{-k}x^1}+a_2e^{-k_3\sqrt{-k}x^1}
\end{aligned}
  \right. \ , \ a_1,a_2,c_2\in \mathbb R,
\end{equation*}

\hspace{0.5cm} (c) $k>0$ and 
\begin{equation*}
\left\{
    \begin{aligned}
&V^1(x^1,x^3)=\frac{\displaystyle \sqrt{k}}{\displaystyle k}\left(\frac{\displaystyle f_1'}{\displaystyle f_1^2}\right)(x^3)\left(a_1\sin({k_3\sqrt{k}x^1})-a_2\cos({k_3\sqrt{k}x^1})\right)\\
&V^2=c_2\\
&V^3(x^1)=a_1\cos({k_3\sqrt{k}x^1})+a_2\sin({k_3\sqrt{k}x^1})
\end{aligned}
  \right. \ , \ a_1,a_2,c_2\in \mathbb R.
\end{equation*}
\end{theorem}

\begin{example}
The vector field $V=c\frac{\displaystyle \partial}{\displaystyle \partial x^1}$, $c\in\mathbb R\setminus\{0\}$, 
is a Killing vector field on the biwarped product manifold
$$\left(\mathbb R^3, \ {g}=\frac{1}{f_1^2}dx^1\otimes dx^1+\frac{1}{f_2^2}dx^2\otimes dx^2+dx^3\otimes dx^3\right),$$
where $$f_1(x^3)=e^{a_1x^3}, \ \ f_2(x^3)=e^{a_2x^3}, \ \ a_1,a_2\in \mathbb R\setminus\{0\}, a_1\neq a_2.$$
\end{example}

\begin{corollary}\label{co2b}
If $f_1=f_2=:f(x^3)$, $f_3=k_3\in \mathbb R\setminus\{0\}$, then $V=\sum_{k=1}^3V^kE_k$ with $V^2=c_2$ is a Killing vector field on $(\mathbb R^3,g)$ if and only if one of the following assertions holds:

(A) \begin{equation*}
\left\{
    \begin{aligned}
&V^1(x^3)=\frac{\displaystyle c}{\displaystyle f(x^3)}\\
&V^2=0\\
&V^3=0
\end{aligned}
  \right. \ , \ c\in \mathbb R;
\end{equation*}

(B) \begin{equation*}
\left\{
    \begin{aligned}
&V^1(x^3)=\frac{\displaystyle c_1}{\displaystyle f(x^3)}\\
&V^2=c_2\\
&V^3=0
\end{aligned}
  \right. \ , \ c_1\in \mathbb R,c_2\in \mathbb R\setminus\{0\},
\end{equation*}
and $f=k_1\in \mathbb R\setminus\{0\}$;

(C) $f=k_1\in \mathbb R\setminus\{0\}$ and
\begin{equation*}
\left\{
    \begin{aligned}
&V^1(x^3)=-\frac{\displaystyle k_1a_1}{\displaystyle k_3}x^3+a_2\\
&V^2=c_2\\
&V^3(x^1)=a_1x^1+a_3
\end{aligned}
  \right. \ , \ c_2,a_1,a_2,a_3\in \mathbb R.
\end{equation*}
\end{corollary}

For the last case, we prove the following result.
\begin{theorem}\label{te2}
Let $f_1=f_1(x^3)$, $f_2=f_2(x^3)$, $f_3=k_3\in \mathbb R\setminus\{0\}$. Then, a vector field $V=\sum_{k=1}^3V^kE_k$ with $V^3=c\in \mathbb R$ is a Killing vector field on $(\mathbb R^3,g)$ if and only if one of the following four assertions holds:

(A) \begin{equation*}
\left\{
    \begin{aligned}
&V^1(x^3)=\frac{\displaystyle c_1}{\displaystyle f_1(x^3)}\\
&V^2(x^3)=\frac{\displaystyle c_2}{\displaystyle f_2(x^3)}\\
&V^3=0
\end{aligned}
  \right. \ , \ c_1,c_2\in \mathbb R;
\end{equation*}

(B) \begin{equation*}
\left\{
    \begin{aligned}
&V^1(x^1,x^3)=\frac{\displaystyle k_3c\bar{c_1}x^1+\hat{c_1}}{\displaystyle {c_1}e^{\bar{c_1}x^3}}\\
&V^2(x^2,x^3)=\frac{\displaystyle k_3c\bar{c_2}x^2+\hat{c_2}}{\displaystyle {c_2}e^{\bar{c_2}x^3}}\\
&V^3=c
\end{aligned}
  \right. \ , \ \bar{c_1}, \bar{c_2}, \hat{c_1}, \hat{c_2}\in \mathbb R, c, {c_1}, {c_2}\in \mathbb R\setminus \{0\},
\end{equation*}
and $f_i(x^3)=c_ie^{\bar{c_i}x^3}$, $i\in \{1,2\}$;

(C) \begin{equation*}
\left\{
    \begin{aligned}
&V^1(x^2,x^3)=\frac{\displaystyle k_0}{\displaystyle f_2(x^3)}x^2+\frac{\displaystyle c_1}{\displaystyle f_1(x^3)}\\
&V^2(x^1,x^3)=-\frac{\displaystyle k_0}{\displaystyle f_1(x^3)}x^1+\frac{\displaystyle c_2}{\displaystyle f_2(x^3)}\\
&V^3=0
\end{aligned}
  \right. \ , \ c_1,c_2\in \mathbb R, k_0\in \mathbb R\setminus \{0\},
\end{equation*}
and $\frac{\displaystyle f_1}{\displaystyle f_2}$ is constant;

(D) \begin{equation*}
\left\{
    \begin{aligned}
&V^1(x^1,x^2,x^3)=\frac{\displaystyle k_0}{\displaystyle c_2e^{\bar{c}x^3}}x^2+\frac{\displaystyle k_3c\bar{c}x^1+\hat{c_1}}{\displaystyle c_1e^{\bar{c}x^3}}+\tilde{c_1}\\
&V^2(x^1,x^2,x^3)=-\frac{\displaystyle k_0}{\displaystyle c_1e^{\bar{c}x^3}}x^1+\frac{\displaystyle k_3c\bar{c}x^2+\hat{c_2}}{\displaystyle c_2e^{\bar{c}x^3}}+\tilde{c_2}\\
&V^3=c
\end{aligned}
  \right.  \ , \ \bar{c},\hat{c_1},\hat{c_2},\tilde{c_1},\tilde{c_2}\in \mathbb R, c, c_1, c_2, k_0\in \mathbb R\setminus \{0\},
\end{equation*}
and $f_i(x^3)=c_ie^{\bar{c}x^3}$, $i\in \{1,2\}$, such that
$\bar{c}(\tilde{c_1}^2+\tilde{c_2}^2)=0$.
\end{theorem}
\begin{proof}
Since $V^3=c\in \mathbb R$, \eqref{s4} becomes
\begin{equation*}
\left\{
    \begin{aligned}
&\frac{\displaystyle \partial K}{\displaystyle \partial x^1}(x^1,x^2)=0\\
&\frac{\displaystyle \partial K}{\displaystyle \partial x^2}(x^1,x^2)=0
    \end{aligned}
  \right. \ ;
\end{equation*}
hence, $K=k_0\in \mathbb R$, and \eqref{s2} and \eqref{s3} imply
\begin{equation*}
\left\{
    \begin{aligned}
&\frac{\displaystyle \partial V^1}{\displaystyle \partial x^2}(x^1,x^2,x^3)=\frac{\displaystyle k_0}{\displaystyle f_2(x^3)}\\
&\frac{\displaystyle \partial V^2}{\displaystyle \partial x^1}(x^1,x^2,x^3)=-\frac{\displaystyle k_0}{\displaystyle f_1(x^3)}
\end{aligned}
  \right. \ ,
\end{equation*}
which, by integration, give
\begin{equation}\label{s9}
\left\{
    \begin{aligned}
&V^1(x^1,x^2,x^3)=\frac{\displaystyle k_0}{\displaystyle f_2(x^3)}x^2+H_1(x^1,x^3)\\
&V^2(x^1,x^2,x^3)=-\frac{\displaystyle k_0}{\displaystyle f_1(x^3)}x^1+H_2(x^2,x^3)
\end{aligned}
  \right. \ ,
\end{equation}
where $H_1=H_1(x^1,x^3)$ and $H_2=H_2(x^2,x^3)$. Now, differentiating the equations of \eqref{s9} and using the 1st, the 6th, the 2nd, and the 5th equations of \eqref{sms}, we find
\begin{equation*}
\left\{
    \begin{aligned}
&\frac{\displaystyle \partial H_1}{\displaystyle \partial x^1}(x^1,x^3)=k_3\left(\frac{\displaystyle f_1'}{\displaystyle f_1^2}\right)(x^3)c\\
&\frac{\displaystyle \partial H_1}{\displaystyle \partial x^3}(x^1,x^3)=k_0\left[\left(\frac{\displaystyle f_2'}{\displaystyle f_2^2}-\frac{\displaystyle f_1'}{\displaystyle f_1f_2}\right)(x^3)\right]x^2-\left(\frac{\displaystyle f_1'}{\displaystyle f_1}\right)(x^3)H_1(x^1,x^3)
\end{aligned}
  \right. \ ,
\end{equation*}
and
\begin{equation*}
\left\{
    \begin{aligned}
&\frac{\displaystyle \partial H_2}{\displaystyle \partial x^2}(x^2,x^3)=k_3\left(\frac{\displaystyle f_2'}{\displaystyle f_2^2}\right)(x^3)c\\
&\frac{\displaystyle \partial H_2}{\displaystyle \partial x^3}(x^2,x^3)=k_0\left[\left(\frac{\displaystyle f_2'}{\displaystyle f_1f_2}-\frac{\displaystyle f_1'}{\displaystyle f_1^2}\right)(x^3)\right]x^1-\left(\frac{\displaystyle f_2'}{\displaystyle f_2}\right)(x^3)H_2(x^2,x^3)
\end{aligned}
  \right. \ ,
\end{equation*}
which imply that either ($i_1$) ($k_0=0$) or ($i_2$) ($k_0\neq 0$ and $\frac{\displaystyle f_1'}{\displaystyle f_1}=\frac{\displaystyle f_2'}{\displaystyle f_2}$).

($i_1$) If $k_0=0$, then $K=0$, and \eqref{s2}, \eqref{s3}, and \eqref{s9}, will consequently imply
\begin{equation*}
\left\{
    \begin{aligned}
&\frac{\displaystyle \partial V^1}{\displaystyle \partial x^2}(x^1,x^2,x^3)=0\\
&\frac{\displaystyle \partial V^2}{\displaystyle \partial x^1}(x^1,x^2,x^3)=0
\end{aligned}
  \right. \ ,
\end{equation*}
which imply that $V^1=V^1(x^1,x^3)$ and $V^2=V^2(x^2,x^3)$;
\begin{equation*}
\left\{
    \begin{aligned}
&\frac{\displaystyle \partial V^1}{\displaystyle \partial x^1}(x^1,x^3)=k_3\left(\frac{\displaystyle f_1'}{\displaystyle f_1^2}\right)(x^3)c\\
&\frac{\displaystyle \partial V^1}{\displaystyle \partial x^3}(x^1,x^3)=-\left(\frac{\displaystyle f_1'}{\displaystyle f_1}\right)(x^3)V^1(x^1,x^3)
\end{aligned}
  \right. \ ,
\end{equation*}
\begin{equation*}
\left\{
    \begin{aligned}
&\frac{\displaystyle \partial V^2}{\displaystyle \partial x^2}(x^2,x^3)=k_3\left(\frac{\displaystyle f_2'}{\displaystyle f_2^2}\right)(x^3)c\\
&\frac{\displaystyle \partial V^2}{\displaystyle \partial x^3}(x^2,x^3)=-\left(\frac{\displaystyle f_2'}{\displaystyle f_2}\right)(x^3)V^2(x^2,x^3)
\end{aligned}
  \right. \ .
\end{equation*}
From the 2nd equations of the last two systems, we infer that
$$V^1(x^1,x^3)=\frac{M_1(x^1)}{f_1(x^3)}, \ \ V^2(x^2,x^3)=\frac{M_2(x^2)}{f_2(x^3)},$$
where $M_1=M_1(x^1)$ and $M_2=M_2(x^2)$, which, by differentiation, give
\begin{equation*}
\left\{
    \begin{aligned}
&M_1'(x^1)=k_3c\left(\frac{\displaystyle f_1'}{\displaystyle f_1}\right)(x^3)\\
&M_2'(x^2)=k_3c\left(\frac{\displaystyle f_2'}{\displaystyle f_2}\right)(x^3)
\end{aligned}
  \right . \ .
\end{equation*}
We deduce the following possible cases: (${i_1}_a$) ($c=0$) and (${i_1}_b$) ($c\neq 0$, $\displaystyle\frac{f_1'}{f_1}=\bar{c}_1\in \mathbb R$, $\displaystyle\frac{f_2'}{f_2}=\bar{c}_2\in \mathbb R$).

In the first case, (${i_1}_a$), we get $M_1=c_3\in \mathbb R$ and $M_2=c_4\in \mathbb R$, therefore,
\begin{equation*}
\left\{
    \begin{aligned}
&V^1(x^3)=\frac{\displaystyle c_3}{\displaystyle f_1(x^3)}\\
&V^2(x^3)=\frac{\displaystyle c_4}{\displaystyle f_2(x^3)}\\
&V^3=0
\end{aligned}
  \right.  \ .
\end{equation*}

In the second case, (${i_1}_b$), we get $f_i(x^3)=c_ie^{\bar{c_i}x^3}$, ${c_i}\in \mathbb R\setminus\{0\}$, $\bar{c_i}\in \mathbb R$, $i\in \{1,2\}$; therefore,
\begin{equation*}
\left\{
    \begin{aligned}
&M_1(x^1)=k_3c\bar{c_1}x^1+\hat{c_1}\\
&M_2(x^2)=k_3c\bar{c_2}x^2+\hat{c_2}
\end{aligned}
  \right. \ , \hat{c_1}, \hat{c_2}\in \mathbb R \ ,
\end{equation*}
\begin{equation*}
\left\{
    \begin{aligned}
&V^1(x^1,x^3)=\frac{\displaystyle k_3c\bar{c_1}x^1+\hat{c_1}}{\displaystyle {c_1}}e^{-\bar{c_1}x^3}\\
&V^2(x^2,x^3)=\frac{\displaystyle k_3c\bar{c_2}x^2+\hat{c_2}}{\displaystyle {c_2}}e^{-\bar{c_2}x^3}\\
&V^3=c
\end{aligned}
  \right. \ .
\end{equation*}

(${i_2}$) If $k_0\neq 0$ and $\frac{\displaystyle f_1'}{\displaystyle f_1}=\frac{\displaystyle f_2'}{\displaystyle f_2}$, then $f_2(x^3)=c_0f_1(x^3)$, $c_0\in \mathbb R\setminus \{0\}$, and
\begin{equation}\label{s10}
\left\{
    \begin{aligned}
&\frac{\displaystyle \partial H_1}{\displaystyle \partial x^1}(x^1,x^3)=k_3c\left(\frac{\displaystyle f_1'}{\displaystyle f_1^2}\right)(x^3)\\
&\frac{\displaystyle \partial H_1}{\displaystyle \partial x^3}(x^1,x^3)=-\left(\frac{\displaystyle f_1'}{\displaystyle f_1}\right)(x^3)H_1(x^1,x^3)
\end{aligned}
  \right. \ ,
\end{equation}
and
\begin{equation}\label{s11}
\left\{
    \begin{aligned}
&\frac{\displaystyle \partial H_2}{\displaystyle \partial x^2}(x^2,x^3)=k_3c\left(\frac{\displaystyle f_2'}{\displaystyle f_2^2}\right)(x^3)\\
&\frac{\displaystyle \partial H_2}{\displaystyle \partial x^3}(x^2,x^3)=-\left(\frac{\displaystyle f_2'}{\displaystyle f_2}\right)(x^3)H_2(x^2,x^3)
\end{aligned}
  \right. \ .
\end{equation}
From the 2nd equations of the last two systems, we infer that
$$H_1(x^1,x^3)=\frac{M_1(x^1)}{f_1(x^3)}, \ \ H_2(x^2,x^3)=\frac{M_2(x^2)}{f_2(x^3)},$$
where $M_1=M_1(x^1)$ and $M_2=M_2(x^2)$, which, by differentiation, give
\begin{equation*}
M_1'(x^1)=k_3c\left(\frac{\displaystyle f_1'}{\displaystyle f_1}\right)(x^3)=k_3c\left(\frac{\displaystyle f_2'}{\displaystyle f_2}\right)(x^3)=M_2'(x^2).
\end{equation*}
We deduce the following possible cases: (${i_2}_a$) ($c=0$) and (${i_2}_b$) ($c\neq 0$, $\displaystyle\frac{f_1'}{f_1}=\displaystyle\frac{f_2'}{f_2}=k\in \mathbb R$).

In the first case, (${i_2}_a$), we get $M_1=c_3\in \mathbb R$ and $M_2=c_4\in \mathbb R$; therefore,
\begin{equation*}
\left\{
    \begin{aligned}
&H_1(x^3)=\frac{\displaystyle c_3}{\displaystyle f_1(x^3)}\\
&H_2(x^3)=\frac{\displaystyle c_4}{\displaystyle f_2(x^3)}
\end{aligned}
  \right. \ ,
\end{equation*}
and, finally,
\begin{equation*}
\left\{
    \begin{aligned}
&V^1(x^2,x^3)=\frac{\displaystyle k_0}{\displaystyle f_2(x^3)}x^2+\frac{\displaystyle c_3}{\displaystyle f_1(x^3)}\\
&V^2(x^1,x^3)=-\frac{\displaystyle k_0}{\displaystyle f_1(x^3)}x^1+\frac{\displaystyle c_4}{\displaystyle f_2(x^3)}\\
&V^3=0
\end{aligned}
  \right. \ .
\end{equation*}

In the second case, (${i_2}_b$), we get $f_1(x^3)=c_1e^{kx^3}$, $c_1\in \mathbb R\setminus\{0\}$, $f_2(x^3)=c_2e^{kx^3}$, $c_2\in \mathbb R\setminus\{0\}$, and
\begin{equation*}
\left\{
    \begin{aligned}
&\frac{\displaystyle \partial H_1}{\displaystyle \partial x^1}(x^1,x^3)=\frac{\displaystyle kk_3c}{\displaystyle c_1}e^{-kx^3}\\
&\frac{\displaystyle \partial H_1}{\displaystyle \partial x^3}(x^1,x^3)=-k\frac{\displaystyle M_1(x^1)}{\displaystyle c_1}e^{-kx^3}
\end{aligned}
  \right. \ ,
\end{equation*}
and
\begin{equation*}
\left\{
    \begin{aligned}
&\frac{\displaystyle \partial H_2}{\displaystyle \partial x^2}(x^2,x^3)=\frac{\displaystyle kk_3c}{\displaystyle c_2}e^{-kx^3}\\
&\frac{\displaystyle \partial H_2}{\displaystyle \partial x^3}(x^2,x^3)=-k\frac{\displaystyle M_2(x^2)}{\displaystyle c_2}e^{-kx^3}
\end{aligned}
  \right. \ .
\end{equation*}
Now, integrating the 1st equations of the previous systems, we find
\begin{equation*}
\left\{
    \begin{aligned}
&H_1(x^1,x^3)=\frac{\displaystyle kk_3c}{\displaystyle c_1}e^{-kx^3}x^1+N_1(x^3)\\
&H_2(x^2,x^3)=\frac{\displaystyle kk_3c}{\displaystyle c_2}e^{-kx^3}x^2+N_2(x^3)
\end{aligned}
  \right. \ ,
\end{equation*}
where $N_1=N_1(x^3)$ and $N_2=N_2(x^3)$. By differentiating them with respect to $x^3$ and using the 2nd equations of the same systems, we get
\begin{equation*}
\left\{
    \begin{aligned}
&N_1'(x^3)=\frac{\displaystyle k^2k_3c}{\displaystyle c_1}e^{-kx^3}x^1-k\frac{\displaystyle M_1(x^1)}{\displaystyle c_1}e^{-kx^3}\\
&N_2'(x^3)=\frac{\displaystyle k^2k_3c}{\displaystyle c_2}e^{-kx^3}x^2-k\frac{\displaystyle M_2(x^2)}{\displaystyle c_2}e^{-kx^3}
\end{aligned}
  \right. \ ,
\end{equation*}
and, by differentiating them with respect to $x^1$ and $x^2$, respectively, we infer that
$$M_1'(x^1)=kk_3c=M_2'(x^2);$$
therefore,
\begin{equation*}
\left\{
    \begin{aligned}
&M_1(x^1)=kk_3cx^1+c_5\\
&M_2(x^2)=kk_3cx^2+c_6
\end{aligned}
  \right. \ , \  c_5, c_6\in \mathbb R,
\end{equation*}
\begin{equation*}
\left\{
    \begin{aligned}
&N_1(x^3)=\frac{\displaystyle c_5}{\displaystyle c_1}e^{-kx^3}+c_7\\
&N_2(x^3)=\frac{\displaystyle c_6}{\displaystyle c_2}e^{-kx^3}+c_8
\end{aligned}
  \right. \ , \ c_7,c_8\in \mathbb R,
\end{equation*}
\begin{equation*}
\left\{
    \begin{aligned}
&H_1(x^1,x^3)=\frac{\displaystyle kk_3cx^1+c_5}{\displaystyle c_1}e^{-kx^3}+c_7\\
&H_2(x^2,x^3)=\frac{\displaystyle kk_3cx^2+c_6}{\displaystyle c_2}e^{-kx^3}+c_8
\end{aligned}
  \right. \ ,
\end{equation*}
and, finally,
\begin{equation*}
\left\{
    \begin{aligned}
&V^1(x^1,x^2,x^3)=\frac{\displaystyle k_0}{\displaystyle c_2}x^2e^{-kx^3}+\frac{\displaystyle kk_3cx^1+c_5}{\displaystyle c_1}e^{-kx^3}+c_7\\
&V^2(x^1,x^2,x^3)=-\frac{\displaystyle k_0}{\displaystyle c_1}x^1e^{-kx^3}+\frac{\displaystyle kk_3cx^2+c_6}{\displaystyle c_2}e^{-kx^3}+c_8\\
&V^3=c
\end{aligned}
  \right. \ .
\end{equation*}
Now, using the 5th and the 6th equations of \eqref{sms}, we find that
$$\left\{
\begin{aligned}
&kc_7=0\\
&kc_8=0
    \end{aligned}
\right. \  .$$

By a direct computation, we find that the converse implication also holds true, hence, the proof is complete.
\end{proof}

\begin{example}
The vector field $V=\sum_{k=1}^3V^kE_k$, where
\begin{equation*}
\left\{
    \begin{aligned}
&V^1(x^1,x^3)=a_1cx^1e^{-a_1x^3}\\
&V^2(x^2,x^3)=a_2cx^2e^{-a_2x^3}\\
&V^3=c
    \end{aligned}
  \right. \ , a_1,a_2, c\in \mathbb R\setminus\{0\}, a_1\neq a_2
\end{equation*}
is a Killing vector field on the biwarped product manifold
$$\left(\mathbb R^3, \ {g}=\frac{1}{f_1^2}dx^1\otimes dx^1+\frac{1}{f_2^2}dx^2\otimes dx^2+dx^3\otimes dx^3\right),$$
for $$f_1(x^3)=e^{a_1x^3}, \ \ f_2(x^3)=e^{a_2x^3}.$$
\end{example}

\begin{example}
The vector field $V=\sum_{k=1}^3V^kE_k$, where
\begin{equation*}
\left\{
    \begin{aligned}
&V^1(x^1,x^2,x^3)=\left(\displaystyle\frac{x^2}{a_2}+\displaystyle\frac{cx^1}{a_1}\right)e^{-x^3}\\
&V^2(x^1,x^2,x^3)=\left(-\displaystyle\frac{x^1}{a_1}+\displaystyle\frac{cx^2}{a_2}\right)e^{-x^3}\\
&V^3=c
    \end{aligned}
  \right. \ , a_1,a_2, c\in \mathbb R\setminus\{0\}, a_1\neq a_2
\end{equation*}
is a Killing vector field on the biwarped product manifold
$$\left(\mathbb R^3, \ {g}=\frac{1}{f_1^2}dx^1\otimes dx^1+\frac{1}{f_2^2}dx^2\otimes dx^2+dx^3\otimes dx^3\right),$$
for $$f_1(x^3)=a_1e^{x^3}, \ \ f_2(x^3)=a_2e^{x^3}.$$
\end{example}

\begin{corollary}\label{co2}
If $f_1=f_2=:f(x^3)$, $f_3=k_3\in \mathbb R\setminus\{0\}$, then $V=\sum_{k=1}^3V^kE_k$ with $V^3=c\in \mathbb R$ is a Killing vector field on $(\mathbb R^3,g)$ if and only if one of the following two assertions holds:

(A) \begin{equation*}
\left\{
    \begin{aligned}
&V^1(x^2,x^3)=\frac{\displaystyle 1}{\displaystyle f(x^3)}(k_0x^2+c_1)\\
&V^2(x^1,x^3)=\frac{\displaystyle 1}{\displaystyle f(x^3)}(-k_0x^1+c_2)\\
&V^3=0
\end{aligned}
  \right. \ , \ c_1,c_2,k_0\in \mathbb R;
\end{equation*}

(B) \begin{equation*}
\left\{
    \begin{aligned}
&V^1(x^1,x^2,x^3)=\frac{\displaystyle 1}{\displaystyle c_0}\left(k_0x^2+k_3c\bar{c}x^1+\hat{c_1}\right)e^{-\bar{c}x^3}+\tilde{c_1}
\\
&V^2(x^1,x^2,x^3)=\frac{\displaystyle 1}{\displaystyle c_0}
\left(-k_0x^1+k_3c\bar{c}x^2+\hat{c_2}\right)e^{-\bar{c}x^3}+\tilde{c_2}
\\
&V^3=c
\end{aligned}
  \right.  \ , \ \bar{c},\hat{c_1},\hat{c_2},\tilde{c_1},\tilde{c_2},k_0\in \mathbb R, c, c_0\in \mathbb R\setminus \{0\},
\end{equation*}
and $f(x^3)=c_0e^{\bar{c}x^3}$, such that
$\bar{c}(\tilde{c_1}^2+\tilde{c_2}^2)=0$.
\end{corollary}

\begin{example}
The vector field $V=\sum_{k=1}^3V^kE_k$, where
\begin{equation*}
\left\{
    \begin{aligned}
&V^1(x^1,x^2,x^3)=\displaystyle\frac{1}{k}(cx^1+k_0x^2)e^{-x^3}\\
&V^2(x^1,x^2,x^3)=\displaystyle\frac{1}{k}(cx^2-k_0x^1)e^{-x^3}\\
&V^3=c
    \end{aligned}
  \right. \ , \ c,k\in \mathbb R\setminus\{0\}, k_0\in \mathbb R,
\end{equation*}
is a Killing vector field on the warped product manifold
$$\left(\mathbb R^3, \ {g}=\frac{1}{f^2}(dx^1\otimes dx^1+dx^2\otimes dx^2)+dx^3\otimes dx^3\right),$$
for $$f(x^3)=ke^{x^3}.$$
\end{example}

\begin{example}
The vector field $V=\sum_{k=1}^3V^kE_k$, where
\begin{equation*}
\left\{
    \begin{aligned}
&V^1(x^1,x^2,x^3)=(kcx^1+k_0x^2)e^{-kx^3}\\
&V^2(x^1,x^2,x^3)=(kcx^2-k_0x^1)e^{-kx^3}\\
&V^3=c
    \end{aligned}
  \right. \ , \ c,k\in \mathbb R\setminus\{0\}, k_0\in \mathbb R,
\end{equation*}
is a Killing vector field on the warped product manifold
$$\left(\mathbb R^3, \ {g}=\frac{1}{f^2}(dx^1\otimes dx^1+dx^2\otimes dx^2)+dx^3\otimes dx^3\right),$$
for $$f(x^3)=e^{kx^3}.$$
\end{example}

\textit{Adara M. Blaga}

\textit{Department of Mathematics}

\textit{West University of Timi\c{s}oara}

\textit{Bd. V. P\^{a}rvan 4, 300223, Timi\c{s}oara, Romania}

\textit{adarablaga@yahoo.com}

\end{document}